\documentclass{amsart}
\usepackage{amsfonts,amssymb,amsmath,amsthm}
\usepackage{url}
\usepackage{enumerate}
\usepackage{amsfonts,epic,amsbsy,amscd,amstext}
\usepackage{amssymb,latexsym,amsmath,amsthm,amscd,mathrsfs}
\usepackage{comment}
\usepackage{color}
\usepackage{mathrsfs}
\usepackage{bbold}

\newtheorem{theorem}{Theorem}[section]

\newtheorem{question}[theorem]{Question}
\newtheorem{lemma}[theorem]{Lemma}
\newtheorem{corollary}[theorem]{Corollary}
\theoremstyle{definition}

\newtheorem{definition}[theorem]{Definition}
\numberwithin{equation}{section}

\def \lt {<}
\def \gt {>}

\newcommand{\zfc}{\mathnormal{\mathsf{ZFC}}}

\newcommand{\ch}{\mathnormal{\mathsf{CH}}}

\newcommand{\covm}{\mathnormal{\mathrm{cov}(\mathcal M)}}
\newcommand{\cofm}{\mathnormal{\mathrm{cof}(\mathcal M)}}
\newcommand{\V}{\mathnormal{\mathbf{V}}}

\DeclareMathOperator{\id}{id}
\DeclareMathOperator{\ran}{ran}

\author[D. Fern\'andez]{David~J. Fern\'andez-Bret\'on}
\address{
Department of Mathematics\\
University of Michigan\\
2074 East Hall, 530 Church Street \\
Ann Arbor, MI 48109-1043, U.S.A.}

\email{djfernan@umich.edu}
\urladdr{http://www-personal.umich.mx/\textasciitilde djfernan/}

\author[M. Hru\v{s}\'ak]{Michael Hru\v{s}\'ak}
\address{
Instituto de Matem\'aticas \\
Universidad Nacional Aut\'onoma de M\'exico \\
\'Area de la Investigaci\'on Cient\'{\i}fica, Circuito Exterior, Ciudad Universitaria \\
Coyoac\'an, 04510, M\'exico, D.F., Mexico}
\email{michael@matmor.unam.mx}
\urladdr{http://www.matmor.unam.mx/\textasciitilde michael/}

\keywords{Ultrafilter, Random forcing, dominating number, parametrized diamond principle, 
topology on the rationals.}
\subjclass[2010]{Primary 03E75; Secondary 54A35, 03E35, 03E17, 54F05.}

\thanks{The first author was partially supported by Postdoctoral Fellowship 
number 263820 from the Consejo Nacional de Ciencia y Tecnolog\'{\i}a 
(CONACyT), Mexico; while the research of the second author was partially 
supported by PAPIIT grant IN 108014 and CONACYT grant 177758.}

\begin{document}

\title[Gruff ultrafilters]{Gruff ultrafilters}

\begin{abstract}
We investigate the question of whether $\mathbb Q$ carries an ultrafilter 
generated by perfect sets (such ultrafilters were called
\textit{gruff ultrafilters} by van Douwen). We prove that one can (consistently) 
obtain an affirmative 
answer to this question in three different ways: by assuming a certain 
parametrized diamond principle, from the cardinal invariant equality 
$\mathfrak d=\mathfrak c$, and in the Random real model.
\end{abstract}
\maketitle

\section{Introduction}

In a 1992 paper, Eric van Douwen~\cite{vandouwen} carried out an investigation about 
certain points in the \v Cech-Stone compactification of $\mathbb Q$ (where 
$\mathbb Q$ is equipped with the 
topology inherited from the Euclidean topology on $\mathbb R$, so that points 
in $\beta\mathbb Q$ can be realized as maximal filters of closed sets), with the 
property that they actually generate an ultrafilter on $\mathbb Q$. In other words, 
van Douwen was looking at ultrafilters over $\mathbb Q$ that have a base of closed 
sets, and among those he paid particular attention to the ones where the elements 
of a base can be taken to be crowded (recall that a set is crowded if it has no
isolated points), in addition to being closed. This was the motivation for stating the 
following definition.

\begin{definition}\label{defgruff}
A nonprincipal ultrafilter $u$ on $\mathbb Q$ is said to be \textbf{gruff} (a pun on 
the fact that these are points in $\beta\mathbb Q$ that ``generate  real ultra 
filters") if it has a base of perfect (i.e. closed and crowded) subsets of 
$\mathbb Q$. This is, we require that $(\forall A\in u)(\exists X\in u)(X\text{ is 
perfect and }X\subseteq A)$.
\end{definition}

Recall that a coideal on a set $X$ is a family $\mathscr A$ with the property 
that $\varnothing\notin\mathscr A$, $\mathscr A$ is closed under supersets, and 
whenever an element $A\in\mathscr A$ is written as $A=A_0\cup A_1$, there 
exists an $i\in 2$ such 
that $A_i\in\mathscr A$. If the infinite set $X$ has a topology in which $X$ itself is 
crowded, then  the family 
\begin{equation*}
 \mathscr C=\{A\subseteq X\big|A\text{ contains an infinite crowded set}\}
\end{equation*}
constitutes a coideal. Moreover, in the topological 
space $\mathbb Q$, every infinite crowded set contains an infinite perfect subset. 
This fact, 
which is not true in a general topological space (for example, in every Polish 
space it is possible to construct \textit{Bernstein 
sets}, sets that are not contained in nor disjoint from any uncountable 
perfect subset), implies that the family
\begin{equation*}
 \mathscr P=\{A\subseteq X\big|A\text{ contains an infinite perfect set}\},
\end{equation*}
also constitutes a coideal on $\mathbb Q$. It is for this reason that 
Definition~\ref{defgruff} is justified.

The main question that van Douwen asked about gruff ultrafilters is whether 
their existence can be proved in $\zfc$. He himself~\cite[Thm. 2.1]{vandouwen} 
provided a partial 
answer by proving that the existence of a gruff ultrafilter follows from 
the cardinal invariant equality $\covm=\mathfrak c$, which is equivalent to 
Martin's Axiom restricted to countable forcing notions. Although the question 
remains open, more partial results have been proven. 
Copl\'akov\'a and Hart~\cite[Thm. 1]{coplakova} 
proved in 1999 that the existence of a gruff 
ultrafilter follows from $\mathfrak b=\mathfrak c$. Some time after, in 
2003, Ciesielski and Pawlikowski~\cite[Thm. 4.22]{ciesielski-pawlikowski} showed that 
the existence of a gruff ultrafilter follows from a combinatorial principle 
known as $\mathnormal{\mathsf{CPA}_{\mathrm{prism}}^{\mathrm{game}}}$, 
which, in particular, implies that there exist gruff ultrafilters in the Sacks 
model, as this model satisfies that combinatorial principle. This theorem was 
improved shortly after by Mill\'an~\cite[Thm. 3]{millan}, who showed that, in fact, 
$\mathnormal{\mathsf{CPA}_{\mathrm{prism}}^{\mathrm{game}}}$ implies the 
existence of a gruff ultrafilter that is at the same time a Q-point (it is shown 
in~\cite[Prop. 5.5.5]{cpabook} that a gruff ultrafilter cannot be a P-point).

In this paper, we obtain three more partial answers to  van 
Douwen's question. The first result involves the theory of diamond principles 
that are par\-amet\-rized by a cardinal invariant, as developed in~\cite{weakdiamond}.
We define a cardinal characteristic $\mathfrak r_P$ that relates naturally to 
perfect subsets of $\mathbb Q$, and show that its corresponding parametrized 
diamond principle $\diamondsuit(\mathfrak r_P)$ implies the existence of a gruff 
ultrafilter which is at the same time a Q-point. We also show that 
this parametrized diamond principle holds in the 
Sacks model, thus providing an alternative proof of Mill\'an's theorem on the 
existence 
of gruff Q-points in this model. Our second result is that the existence 
of gruff ultrafilters follows from the cardinal invariant equality 
$\mathfrak d=\mathfrak c$. Since (it is provable in $\zfc$ that) 
$\mathfrak d\geq\mathfrak b$ and $\mathfrak d\geq\covm$, but both inequalities 
can be consistently strict (even simultaneously), this result is stronger than 
both van Douwen's, and Copl\'akov\'a and Hart's. Finally, our third result is that 
in the Random real model there exists a gruff ultrafilter (this, together with the 
$\diamondsuit(\mathfrak r_P)$ result mentioned above,
shows that the existence of gruff ultrafilters is consistent with 
$\mathfrak d\lt\mathfrak c$). First we will prove a lemma 
that will simplify our interaction with gruff 
ultrafilters. Throughout this paper, the notation $(a,b)$ will 
be used to refer to intervals on $\mathbb Q$. In other words, 
$(a,b)=\{q\in\mathbb Q\big|a\lt q\lt b\}$ whenever $a,b\in\mathbb R$.

\begin{lemma}\label{boundedvsunbounded}
There exists a gruff ultrafilter on $\mathbb Q$ if and only if there 
exists an ultrafilter on the set of positive rational numbers $\mathbb Q^+$
with a base of perfect unbounded sets.
\end{lemma}

\begin{proof}
The ``if" part is obvious. For the other direction, assume that 
$u$ is a gruff ultrafilter on $\mathbb Q$. Without 
loss of generality, $u$ concentrates on $\mathbb Q^+$, since otherwise 
it would have to concentrate on the set of negative rational numbers 
$\mathbb Q^-$, in which case the function $x\longmapsto -x$ (which 
is an autohomeomorphism of $\mathbb Q$) will map 
$u$ to another gruff ultrafilter $-u$ that concentrates on $\mathbb Q^+$.

Hence we get to assume that $u$ is an ultrafilter on $\mathbb Q^+$ with 
a base of perfect sets. If all of these perfect sets happen to be 
unbounded, we are done. Otherwise $u$ contains a bounded set, 
and hence $u$ must converge to some real number $r$. If $u$ concentrates 
on $(0,r)$, then take an order-isomorphism 
$f:(0,r)\longrightarrow\mathbb Q^+$, and notice that $f$ will 
be a homeomorphism mapping every element of $u$ to an unbounded set. 
Otherwise, $u$ must concentrate on $(r,\infty)$, and so we can take 
an order anti-isomorphism $f:(r,\infty)\longrightarrow\mathbb Q^+$, 
which will map every element of $u$ to an unbounded set. In either 
case, the ultrafilter $f(u)$ will have a base of perfect unbounded sets.
\end{proof}

Throughout the rest of this paper, we will 
write $\mathbb Q$ instead of $\mathbb Q^+$. Similarly, $\mathscr P$ will now denote 
the coideal of subsets of $\mathbb Q$ (i.e. $\mathbb Q^+$) containing 
an \textit{unbounded} 
perfect set, and we will take gruff to mean an ultrafilter on $\mathbb Q$ 
(i.e. $\mathbb Q^+$) with a base of sets that are both perfect and unbounded. 
Lemma~\ref{boundedvsunbounded} guarantees that these changes do not 
make any difference regarding the question of the existence of gruff 
ultrafilters.

\section{A cardinal invariant and its parametrized diamond principle}

We will start by laying down some terminology and results from~\cite{weakdiamond} 
that we will be using throughout this section. The theory of 
parametrized diamond principles involves cardinal 
invariants given by triples $(A,B,E)$ where $|A|\leq\mathfrak c$, 
$|B|\leq\mathfrak c$ and $E\subseteq A\times B$, where we additionally 
require that $(\forall a\in A)(\exists b\in B)(a\ E\ b)$ and that 
$(\forall b\in B)(\exists a\in A)\neg(a\ E\ b)$ (these last two requirements 
are only to ensure existence and nontriviality of the corresponding 
cardinal invariant). The cardinal invariant associated to such a triple 
$(A,B,E)$ (sometimes referred to as the \textit{evaluation} of the triple), 
denoted by $\langle A,B,E\rangle$, is given by
\begin{equation*}
\langle A,B,E\rangle=\min\{|X|: X\subseteq B\wedge(\forall a\in A)(\exists b\in X)(a\ E\ b)\}.
\end{equation*}

We consider cardinal invariants $\langle A,B,E\rangle$ that are Borel, which we take to mean 
that all three entries of the 
corresponding triple $(A,B,E)$, can be viewed as Borel subsets of some Polish space. Then 
the \emph{parametrized diamond principle} of the triple is the statement 
that for every Borel function 
$F:2^{<\omega_1}\longrightarrow A$ (where Borel means that for every $\alpha<\omega_1$ 
the restriction $F\upharpoonright2^{<\alpha}$ is a Borel function) there exists a 
$g:\omega_1\longrightarrow B$ such that for every $f:\omega_1\longrightarrow 2$, it is the case 
that $F(f\upharpoonright\alpha)\ E\ g(\alpha)$ for stationarily many $\alpha<\omega_1$. This 
statement is denoted by $\diamondsuit(A,B,E)$. The fundamental theorem regarding these 
parametrized diamond principles is the following

\begin{theorem}[\cite{weakdiamond}, Thm. 6.6]\label{fundamental}
 Let $(A,B,E)$ be a Borel triple defining a cardinal invariant. Let 
 $\mathbb P$ be a Borel forcing notion such that 
 $\mathbb P\simeq 2^+ \times \mathbb P$, and let 
 $\langle\mathbb P_\alpha,\mathring{\mathbb Q_\alpha}\big|\alpha­<\omega_2\rangle$ be 
 a countable support iteration such that for every $\alpha$, 
 $\mathbb P_\alpha\Vdash``\mathring{\mathbb Q_\alpha}=\mathbb P"$. Assume 
 also that the final step of the iteration, $\mathbb P_{\omega_2}$, is proper. Then 
 \begin{equation*}
  \mathbb P_{\omega_2}\Vdash``\diamondsuit(A,B,E)"\text{ if and only if 
 }\mathbb P_{\omega_2}\Vdash``\langle A,B,E\rangle\leq\omega_1".
 \end{equation*}
\end{theorem}

We will say that a set $X$ is \emph{two-sided crowded} if 
for every $x\in X$ and every $\varepsilon\gt0$ there are points $y,z\in X$ such 
that $x-\varepsilon\lt y\lt x\lt z\lt x+\varepsilon$. A two-sided perfect 
set would be a closed, two-sided crowded set. In what follows, we will say  
that $X\subseteq\mathbb Q$ is scattered if no subset of $X$ is two-sided 
crowded. We let
\begin{equation*}
\mathscr B=\{X\subseteq\mathbb Q\big|X\text{ is two-sided perfect and unbounded}\},
\end{equation*}
and note that $\mathscr B$ generates a coideal $\mathscr P'\subseteq\mathscr P$.

We consider here the cardinal invariant $\mathfrak r_P$ to be 
the evaluation of the triple $(\mathcal P (\mathbb Q),\mathscr B,R)$, where 
the relation ``$Y\ R\ X$" ($X$ \emph{reaps} $Y$ modulo $\mathscr P'$, or 
$Y$ \emph{is reaped by} $X$ modulo $\mathscr P'$) means that either 
$X\setminus Y\notin\mathscr P'$ or $X\cap Y\notin\mathscr P'$ ($X$ is 
either contained in, or disjoint from, $Y$, modulo 
$\mathscr P'$). In other words, 
$\mathfrak r_P=\mathfrak r(\mathcal P\left(\mathbb Q\times\omega)/({\mathsf{scattered}\times\mathsf{fin}})\right)$. Equivalently, $\mathfrak r_P$ is the least cardinality of a family $\mathscr X$ of unbounded two-sided 
perfect subsets of $\mathbb Q$ such that for every colouring of the elements 
of $\mathbb Q$ into two colours, there exists an element of $\mathscr X$ 
which is monochromatic, except possibly for a scattered or a bounded subset.

Hence the 
combinatorial principle 
$\diamondsuit(\mathfrak r_P)$ is the statement that for every Borel function 
$F:2^{<\omega_1}\longrightarrow\mathcal P (\mathbb Q)$, there exists a function 
$g:\omega_1\longrightarrow\mathscr B$ (this is, an $\omega_1$-sequence of 
 two-sided perfect unbounded 
subsets of $\mathbb Q$) satisfying that for every $f:\omega_1\longrightarrow2$, 
$g(\alpha)$ will reap $F(f\upharpoonright\alpha)$ modulo $\mathscr P'$ for 
stationarily many $\alpha<\omega_1$.

In order to use the combinatorial principle $\diamondsuit(\mathfrak r_P)$, 
we need a definition and a lemma.

\begin{definition}
If $\mu$ is an ordinal, we say that a $\mu$-sequence 
$\langle X_\alpha\big|\alpha<\mu\rangle$ of subsets of $\mathbb Q$ is \emph{descending 
modulo} $\mathscr P'$ if every $X_\alpha\in\mathscr P'$ and, whenever $\xi<\alpha<\mu$, 
we have that $X_\alpha\setminus X_\xi\notin\mathscr P'$.
\end{definition}

\begin{lemma}\label{pseudoint}
Let $\langle X_n\big|n<\omega\rangle$ be a descending $\omega$-sequence modulo 
$\mathscr P$. Then it is possible to choose, in a Borel way, a two-sided 
perfect
set $X\in\mathscr P$ 
that is almost contained in every $X_n$ modulo $\mathscr P'$ (this is, 
$X\setminus X_n\notin\mathscr P'$).
\end{lemma}

\begin{proof}
First define, for $n<\omega$, the interval $I_n=\left(n\sqrt{2},(n+1)\sqrt{2}\right)$. Then each 
of the $I_n$ will be a clopen interval in $\mathbb Q$, and 
$\mathbb Q=\bigcup_{n<\omega}I_n$. Note that, for every $n<\omega$, since 
$X_n\in\mathscr P'$ then we must have that $X_n\cap I_n$ is not scattered for infinitely 
many $n<\omega$. Similarly, for $n<m<\omega$, we have that $(X_m\setminus X_n)\cap I_k$ 
must be scattered for almost all $k$. Hence we recursively construct an increasing sequence 
$\langle k_n\big|n<\omega\rangle$, and non-scattered sets $B_n\subseteq I_{k_n}$, as 
follows: $k_0$ is any number such that $I_{k_0}\cap X_0$ is not scattered, and we 
let $B_0=I_{k_0}\cap X_0$. Suppose we have 
picked $k_0,\ldots,k_n$ and $B_0,\ldots,B_n$ such that $B_j\setminus X_i$ is scattered 
whenever $i\leq j\leq n$. Then we pick a $k_{n+1}>k_n$ such that 
$(X_{n+1}\setminus X_i)\cap I_{k_{n+1}}$ is scattered whenever $i\leq n$, and we 
 let $B_{n+1}=X_{n+1}\cap I_{k_{n+1}}$. 

Now we choose, in a Borel way, a two-sided perfect 
subset $P_n\subseteq B_n$. To do this, we first fix $C_n$ to 
be the maximal two-sided crowded subset of $B_n$ (which is the 
union of all two-sided crowded subsets of $B_n$). Now we fix an effective 
enumeration $\{q_k\big|k\lt\omega\}$ of $\mathbb Q$, and recursively 
define finite sets $F_k\subseteq B_n$ (for all $k\lt\omega$) and 
clopen intervals $I_k$ (for $q_k\notin C_n$), as follows: at stage $k$, 
we first choose, for each $x\in F_{k-1}$, the least-indexed $y,z\in C_n$ 
which are within $\frac{1}{2^k}$ of $x$, with $y<x<z$, and which 
do not belong to any $I_i$ 
for $i<k$, and put all those $x,y,z$, into $F_k$. Afterwards, 
if $q_k\notin C_n$ then we  let $I_k$ be a clopen interval centred 
around $q_k$ which does not intersect $F_k$ (and otherwise we do 
not define $I_k$). This way, in the end we get the two-sided perfect set 
$P_n=\bigcup\limits_{k<\omega}F_k\subseteq C_n$, which is two-sided crowded by 
construction, and closed because its complement is exactly 
$\bigcup\limits_{k<\omega \atop q_k\notin C_n}I_k$.

In the end, we  define $X=\bigcup_{n<\omega}P_n$, and we are done. 
\end{proof}

\begin{theorem}\label{diamondimpliesgruff}
$\diamondsuit(\mathfrak r_P)$ implies the existence of a gruff ultrafilter.
\end{theorem}

\begin{proof}
 By suitable coding, we consider elements of $2^{<\omega_1}$ that represent pairs 
 $\langle\vec{A},A\rangle$ such that $A\subseteq\mathbb Q$ and 
 $\vec{A}=\langle A_\xi\big|\xi<\alpha\rangle$ is a sequence of two-sided 
 perfect unbounded 
 subsets of 
 $\mathbb Q$ that is descending modulo $\mathscr P$. We 
 choose an increasing sequence $\langle\alpha_n\big|n<\omega\rangle$, cofinal 
 in $\alpha$, and we build, in a Borel way, a two-sided perfect unbounded set 
 $B\subseteq\mathbb Q$ satisfying 
 that $(\forall n<\omega)(B\setminus A_{\alpha_n}\notin\mathscr P)$, using 
 Lemma~\ref{pseudoint}. Since $B$ is two-sided crowded, it is a countable 
 linear order with no endpoints, which means that we can map $B$ homeomorphically to 
 $\mathbb Q$ in a Borel way (by constructing an order-isomorphism between $B$ and 
 $\mathbb Q$ in the usual back-and-forth way) and define 
 $F(\langle\vec{A},A\rangle)$ to be the either the image of $B\cap A$ under 
 this mapping, if that image belongs to $\mathscr P'$, or the image of 
 $B\setminus A$ otherwise (in this construction, if $\alpha$ is a successor 
 cardinal, $\alpha=\xi+1$, then there is no need to pick a cofinal sequence, 
 and we can let $B=A_\xi$ and perform the rest of the construction in the 
 exact same way). Use $\diamondsuit(\mathfrak r_P)$ to 
 get a $g:\omega_1\longrightarrow\mathscr B$ satisfying that for every 
 $f:\omega_1\longrightarrow2$, $F(f\upharpoonright\alpha)$ is reaped, modulo 
 $\mathscr P'$, by the 
 perfect unbounded set $g(\alpha)$ for stationarily many $\alpha<\omega_1$. 
 We use $g$ to recursively 
 construct our gruff ultrafilter.
 
 So assume that we have constructed a sequence of two-sided
 perfect unbounded sets 
 $\langle X_\xi\big|\xi<\alpha\rangle$ that is descending modulo $\mathscr P'$. 
 Using the same cofinal sequence $\langle\alpha_n\big|n<\omega\rangle$ as in the 
 previous paragraph, and in the exact same Borel way, construct a two-sided 
 perfect $X$ which is 
 contained, modulo $\mathscr P'$, in each $X_{\alpha_n}$, and map it 
 homeomorphically onto $\mathbb Q$. We let $X_\alpha\subseteq X$ be the preimage of 
 $g(\alpha)$ under that homeomorphism, so that $X_\alpha$ is a perfect unbounded 
 subset of $\mathbb Q$ (if $\alpha=\xi+1$ then we let 
 $X=X_\xi$ and perform the rest of the construction in the exact same way). This 
 gives us an $\omega_1$-sequence of perfect unbounded sets, descending modulo 
 $\mathscr P'$, $\langle X_\alpha\big|\alpha<\omega_1\rangle$. Clearly the members 
 of this sequence generate a filter, which then will be  gruff provided this filter 
 is an ultrafilter. To see that the filter generated by 
 $\{X_\alpha\big|\alpha<\omega_1\}$ is indeed an ultrafilter, let $A\subseteq\mathbb Q$ 
 and let $f:\omega_1\longrightarrow2$ be the branch of $2^{\omega_1}$ that represents 
 $(\vec{X},A)$ under the relevant coding. Then by choice of $g$, for stationarily 
 many $\alpha<\omega_1$ we will have that $F(\vec{X}\upharpoonright\alpha,A)$ is reaped 
 modulo $\mathscr P$ by $g(\alpha)$. But $F(\vec{X}\upharpoonright\alpha,A)$ is
 the image of either $X\cap A$ or of $X\setminus A$ under the Borel homeomorphism 
 onto $\mathbb Q$ that was obtained at stage $\alpha$, where $X$ is the 
 set obtained at that stage using Lemma~\ref{pseudoint}; and $X_\alpha$ is  
 the preimage of $g(\alpha)$ under that same homeomorphism. Hence we can conclude 
 that $X_\alpha$ reaps $A$ modulo $\mathscr P'$. 
 Since the filter generated by $\{X_\alpha\big|\alpha<\omega_1\}$ consists only 
 of elements in $\mathscr P'$, then we can conclude that such a filter is an 
 ultrafilter, and we are done.
\end{proof}

In order for the previous theorem to be of any use, we need to exhibit models 
where $\diamondsuit(\mathfrak r_P)$ holds. Recall that 
by~\cite[Thm. 6.6]{weakdiamond}, in many of the models of Set Theory that are 
obtained via countable support iterations of proper 
forcing notions, we will have that $\diamondsuit(\mathfrak r_P)$ holds 
if and only if $\mathfrak r_P=\omega_1$.  The following theorem gives 
some bounds for this cardinal invariant. First, we  consider the cardinal invariant $\mathfrak r_\mathbb Q$, introduced 
in~\cite{balcar-hrusak-hernandez}. This invariant is the reaping 
number of the Boolean algebra $\mathcal P(\mathbb Q)/\mathsf{nwd}$, equivalently, 
the minimal size of a family of somewhere dense subsets of $\mathbb Q$ such that 
every set $A\subseteq\mathbb Q$ either contains or is disjoint from an element of 
the family. It is known~\cite[Thm. 3.6]{balcar-hrusak-hernandez} that 
$\max\{\mathfrak r, \cofm\}\leq \mathfrak r_\mathbb Q\leq\mathfrak i$.

\begin{theorem}\label{boundforperfect}
\begin{equation*}
\max\{\mathfrak r,\mathfrak d\}\leq\mathfrak r_P\leq\mathfrak r_\mathbb Q.
\end{equation*}
\end{theorem}

\begin{proof}
To see that $\mathfrak r\leq \mathfrak r_P$, note that if $\mathcal R$ is a family 
of perfect subsets of $\mathbb Q$ witnessing the definition of $\mathfrak r_P$, 
and if $\{U_n: n\in\omega\}$ is a fixed basis for the topology of $\mathbb Q$ 
consisting  of clopen sets, then one can define 
the family
\begin{equation*}
 \mathcal R'=\{ P\cap U_n: P\cap U_n \neq\emptyset\}.
\end{equation*}
Then $\mathcal R'$ is again a family of perfect sets, has the same size as $\mathcal R$, 
and is reaping, i.e. for every $Y\subseteq  \mathbb Q$ there is a 
$P\in\mathcal R'$ such that $P\subseteq Y$ or $P\cap Y=\emptyset$. Thus 
$\mathcal R'$ is a witness to $\mathfrak r$.

Now let us  see that
$\mathfrak r_P\geq\mathfrak d$. Let $\{P_\alpha\big|\alpha<\kappa\}$ be a 
family of perfect sets, and we will argue that if $\kappa<\mathfrak d$ then this family 
cannot be reaping. For each $n<\omega$ let 
$I_n=\left(n\sqrt{2},(n+1)\sqrt{2}\right)$, so that $\mathbb Q=\bigcup_{n<\omega}I_n$. 
Now enumerate $I_n=\{q_{n,k}\big|k<\omega\}$, and for each $\alpha<\kappa$ we will 
define $f_\alpha:\omega\longrightarrow\omega$ as follows: For every $n<\omega$, 
assuming we have defined $f_\alpha(i)$ for all $i<n$, we look at the least $k\geq n$ 
such that $I_k\cap P_\alpha\neq\varnothing$ and we let $j$ be the least such that 
$q_{k,j}\in I_k\cap P_\alpha$. Then we define 
$f_\alpha(n)=f_\alpha(n+1)=\cdots=f_\alpha(k)=j$. Now if $\kappa<\mathfrak d$, we 
can find an $f$ that is not dominated by any $f_\alpha$. Without loss of generality 
$f$ is increasing, so if we let $A=\{q_{i,j}\big|j\leq f(i)\}$ then for every 
$\alpha<\kappa$, the fact that $f_\alpha(k)\leq f(k)$ for infinitely many $k$ implies 
that $P_\alpha\cap A$ is infinite, whereas $A\cap I_n$ is finite for every $n<\omega$ 
so $P_\alpha\setminus A$ must be in $\mathscr P'$ and therefore also be infinite. 
Hence no $P_\alpha$ reaps $A$.

 To see that $\mathfrak r_P\leq\mathfrak r_\mathbb Q$ take  a family $\mathcal R$, 
of size $\mathfrak r_\mathbb Q$, of 
somewhere dense subsets of $\mathbb Q$ such that every set $A\subseteq\mathbb Q$ 
is reaped by an element of $\mathcal R$. For each 
element $X\in\mathcal R$, pick a perfect subset $P_X\subseteq X$. Then the family 
$\mathcal R'=\{P_X\big|X\in\mathcal R\}$ is a witness to $\mathfrak r_P$.
\end{proof}

We now briefly explain how to modify the above construction in order 
to ensure that our gruff ultrafilter is also a Q-point. If we are assuming that 
$\diamondsuit(\mathfrak r_P)$ holds, then by~\cite[Prop. 2.5]{weakdiamond} 
it follows that $\mathfrak r_P=\omega_1$ and so by Theorem~\ref{boundforperfect}, we 
can conclude that $\mathfrak d=\omega_1$. Now by~\cite[Thm. 2.10]{blass}, this 
means that there is a sequence $\langle\mathcal I_\alpha\big|\alpha<\omega_1\rangle$ 
of partitions of $\omega$ into intervals which is dominating, i.e. for every partition 
$\mathcal J=\{J_n\big|n<\omega\}$ of $\omega$ into intervals (we will always 
assume our partitions into intervals to be ordered increasingly), there exists an 
$\alpha<\omega_1$ such that $\mathcal J$ is dominated by $\mathcal I_\alpha$, which 
means that if $\mathcal I_\alpha=\{I_\alpha^n\big|n<\omega\}$ then 
$(\forall^\infty n<\omega)(\exists k<\omega)(J_k\subseteq I_\alpha^n)$. It is 
easily seen that this implies that for every partition $\mathcal J=\{J_n\big|n<\omega\}$ 
of $\omega$ into intervals , there is 
$\alpha<\omega_1$ such that 
$(\forall^\infty n­<\omega)(\exists k<\omega)(J_n\subseteq I_\alpha^k\cup I_\alpha^{k+1})$. 
This clearly implies that, if $u$ is an ultrafilter satisfying 
$(\forall\alpha<\omega_1)(\exists X\in u)(\forall n<\omega)(|X\cap I_\alpha^n|\leq 1)$, 
then $u$ is a Q-point. Thus, if we modify the construction in 
the proof of Theorem~\ref{diamondimpliesgruff} to ensure that each 
$X_\alpha$ thus constructed satisfies $(\forall n<\omega)(|X\cap I_\alpha^n|\leq 1)$, 
then we will succeed in ensuring that our gruff ultrafilter is also a 
Q-point. But notice that each $X_\alpha$ is constructed to be a subset of some two-sided 
perfect set $X$ whose existence is guaranteed by invoking Lemma~\ref{pseudoint}, so 
all we need to do is to modify the proof of this lemma to ensure that, if we are given a partition 
$\mathcal I=\{I_n\big|n<\omega\}$ of $\omega$ into intervals, the two-sided perfect 
set $X$ obtained by this the lemma satisfies 
$(\forall n<\omega)(|X\cap I_n|\leq 1)$. This is very easy to do, since the elements of 
the set $X$ are successively chosen to be the least-indexed ones satisfying certain 
condition, so it suffices to also require them to not belong to any $I_n$ to which some 
formerly chosen element of $X$ belongs. It follows that, by introducing small modifications 
in the proofs of Lemma~\ref{pseudoint} and Theorem~\ref{diamondimpliesgruff}, we can 
construct a gruff ultrafilter that is at the same time a Q-point.

\begin{corollary}
There are gruff Q-points in the Sacks model.
\end{corollary}

\begin{proof}
Since  $\mathfrak r_\mathbb Q=\mathfrak i=\omega_1$ in Sacks model, Theorem~\ref{boundforperfect} 
together with Theorem~\ref{fundamental}
implies that $\diamondsuit(\mathfrak r_P)$ holds in this model, so the 
result follows from Theorem~\ref{diamondimpliesgruff} together 
with the observations outlined on the previous paragraph.
\end{proof}

\begin{question} Is either of the two inequalities in Theorem~\ref{boundforperfect} 
consistently strict?
\end{question}

\section{There are gruff ultrafilters if $\mathfrak d=\mathfrak c$}

In this section we will obtain the existence of gruff ultrafilters from 
the cardinal invariant assumption that $\mathfrak d=\mathfrak c$. In order to 
do this, we will first lay down some notation and preliminary results that 
are central to our proof, and that will also be relevant also for the next 
section.

We describe a method that, given a function $f:\omega\longrightarrow\omega$, 
allows us to use $f$ to construct a perfect subset of every $X\in\mathscr P$. 
From now on we will fix an effective enumeration 
$\{q_n\big|n<\omega\}$ of 
$\mathbb Q$. Now, given $f:\omega\longrightarrow\omega$, define 
the following clopen subsets of $\mathbb Q$:
\begin{equation*}
J_n^f=\left(q_n-\frac{\sqrt{2}}{k},q_n+\frac{\sqrt{2}}{k}\right),
\end{equation*}
where $k$ is the least possible natural number that ensures $q_m\notin J_n^f$ for every $n\neq m\leq f(n)$. Thus $J_n^f$ is a clopen interval, centred at $q_n$, which is just barely small enough so that it does not contain any of the finitely many $q_m$ with $m\neq n$ and $m\leq f(n)$. 
Clearly, the faster the function $f$ grows, the smaller the interval $J_n^f$ will be. In 
other words, if $f(n)\geq g(n)$ then $J_n^f\subseteq J_n^g$. Note also that 
$J_n^f$ is completely 
determined by the single value $f(n)$.

Now, given any subset $X\subseteq\mathbb Q$, we define
\begin{equation*}
 X(f)=\mathbb Q\setminus\left(\bigcup_{q_n\notin X}J_n^f\right)\subseteq X.
\end{equation*}

Then $X(f)$ is a closed set, since it is the complement of an open set.
Some properties of the sets $X(f)$ that we will need later on, and 
that are easy to check, are the following: 
\begin{itemize}
\item If $X\subseteq Y$, then $X(f)\subseteq Y(f)$,
\item $X(f)\cap Y(f)=(X\cap Y)(f)$, and similarly for intersections of any 
 finite amount of sets,
\item If $g\leq f$, then $X(g)\subseteq X(f)$,

\item If $f$ is unbounded, then every rational number belongs to 
 only finitely many of the $J_n^f$. In fact, if $f$ is strictly increasing then 
 the rational number $q_n$ can only belong to at most $n$ of the $J_m^f$, since 
 in this case, whenever $m\gt n$ we have that $n\leq f(n)\lt f(m)$ and therefore 
 $q_n\notin J_m^f$.
\end{itemize}

We wish the set  $X(f)$ to be  not only 
closed, but also perfect. Of course, we can only hope to achieve this 
when $X\in\mathscr P$.

\begin{lemma}\label{specialfunction}
 For every crowded unbounded subset $X\subseteq\mathbb Q$, there exists a function 
$f_X$ such that, whenever $g$ is an increasing function with $g\not\leq^*f_X$, 
$X(g)$ is crowded unbounded (and hence perfect unbounded).
\end{lemma}

\begin{proof}
 Recursively define $f_X$ by 
 $f_X(0)=\min\{k\lt\omega\big|k\gt 0\text{ and }q_k\in X\}$; and once 
 $f_X(n-1)$ has been defined, for every $m\lt n$ with $q_m\in X$
 we define the sets
 \begin{equation*}
  A_n^m=\{k\lt\omega\big|f_X(n-1)\lt k\text{, }q_k\in X\text{, }
  |q_m-q_k|\lt\frac{1}{2^n}\text{ and }
  q_k\notin\bigcup_{m\lt l\lt n}J_l^{\id}\},
 \end{equation*}
 where $\id:\omega\longrightarrow\omega$ is  the identity 
 function. We also define the set 
 \begin{equation*}
  A_n=\{k\lt\omega\big|f_X(n-1)\lt k\text{, }q_k\in X\text{, }
  q_k\gt n\text{ and }
  q_k\notin\bigcup_{l\lt n}J_l^{\id}\}{\color{cyan}.}
 \end{equation*}
 The $A_n^m$ are nonempty because $X$ is crowded, and 
 $A_n$ is nonempty because $X$ is unbounded. So we can let 
 \begin{equation*}
  f_X(n)=\max(\{\min(A_n^m)\big|m\lt n\text{ and }q_m\in X\}\cup\{\min(A_n)\}).
 \end{equation*}
 This finishes the definition of $f_X$.
Now let $g$ be an increasing function such that $g\not\leq^*f_X$. Note that, since $g$ 
is increasing, we have $\id\leq g$, and so $J_n^g\subseteq J_n^{\id}$ for every 
$n\lt\omega$. To prove that $X(g)$ is unbounded, let $m\in\mathbb N$, and we will find 
an element 
$x\in X(g)$ with $x\gt m$. To that end, we fix an $n\gt m$ such that $f_X(n)\lt g(n)$. 
Letting $k=\min(A_n)$, by definition 
we have that $x=q_k\in X$, $q_k\gt n\gt m$ and $q_k\notin\bigcup\limits_{l\lt n}J_l^{\id}$, 
and therefore 
$q_k\notin\bigcup\limits_{l\lt n}J_l^g$. Now since $k\leq f_X(n)\lt g(n)$, we also have 
that $q_k\notin J_l^g$ for $l\geq n$, as long as $l\neq k$. But in particular, 
$q_k\notin J_l^g$ whenever $q_l\notin X$, thus $x=q_k\in X(g)$.

Now to prove that $X(g)$ is crowded, we pick an arbitrary $q_m\in X(g)$ and an 
$\varepsilon\gt 0$, and try to find some $x\in X(g)$ with $x\neq q_m$ and 
$|x-q_m|\lt\varepsilon$. Note that, since $q_m\in X(g)$, then we must have in particular 
that $q_m\notin J_l^g$ whenever $l\lt m$. We let $N$ be large enough that 
$\frac{1}{2^N}\lt\varepsilon$ and  
$(q_m-\frac{1}{2^N},q_m+\frac{1}{2^N})\cap J_l^g=\varnothing$ for all $l\lt m$. 
We now pick an $n\gt N$ such that $f_X(n)\lt g(n)$, and let $k=\min(A_n^m)$. 
Then by definition we have that $x=q_k\in X$, $|q_m-q_k|\lt\frac{1}{2^n}\lt\varepsilon$, 
and 
$q_k\notin J_l^{\id}\supseteq J_l^g$ for all $m\lt l\lt n$. Also, since 
$k\leq f_X(n)\lt g(n)$, we will have $q_k\notin J_l^g$ for all $l\geq n$, as long as 
$l\neq k$. Moreover, our choice of $n$  (and $N$) also ensures that 
$q_k\notin J_l^n$ for $l\lt m$. 
In other words, the only cases where we might have $q_k\in J_l^n$ would be 
when $l$ is either equal to $m$, or equal to $k$; but since $q_k,q_m\in X$, we 
conclude that $k=q_k\in X(g)$.
\end{proof}

This lemma will be of key importance for the rest of 
this article. We will first use it for proving the following result, which 
strengthens both the theorem of Copl\'akov\'a and Hart (which 
assumes $\mathfrak b=\mathfrak c$), and the one of van Douwen (which 
assumes $\covm=\mathfrak c$).

\begin{theorem}
If $\mathfrak d=\mathfrak c$, then there exists a gruff ultrafilter.
\end{theorem}

\begin{proof}
Let $\{A_\alpha\big|\alpha<\mathfrak c\}$ 
 be an enumeration of all subsets of $\mathbb Q$. We will recursively construct 
 sets $X_\alpha\subseteq\mathbb Q$, satisfying the following conditions for every 
 $\alpha<\mathfrak c$:
 \begin{enumerate}
  \item $X_\alpha$ is a perfect unbounded subset of $\mathbb Q$,
  \item $X_\alpha$ is either contained in or disjoint from $A_\alpha$, and
  \item the family $\{X_\xi\big|­\xi\leq\alpha\}$ generates a filter all 
   of whose elements belong to $\mathscr P$ (i.e. they contain 
   a crowded unbounded set).
 \end{enumerate}
 So suppose that we have already constructed $\{X_\xi\big|\xi<\alpha\}$ 
 satisfying all three conditions. Since $\mathscr P$ is a coideal, we 
 can choose an $A\in\{A_\alpha,\mathbb Q\setminus A_\alpha\}$ such that 
 the family $\{X_\xi\big|\xi\lt\alpha\}\cup\{A\}$ generates a 
 filter with all elements belonging to $\mathscr P$. If we now let 
 $B$ be the maximal crowded unbounded subset of $A$ (this $B$ can be seen 
 as either the union of all crowded subsets of $A$, or the part of $A$ that 
 remains after performing the Cantor-Bendixson process), we can see that 
 $\{X_\xi\big|\xi\lt\alpha\}\cup\{B\}$ still generates a filter all of 
 whose elements belong to $\mathscr P$.

 For each finite $F\subseteq\{X_\xi\big|\xi\lt\alpha\}$, we take the 
 maximal crowded unbounded subset $B_F\subseteq\left(\bigcap F\right)\cap B$ (since by hypothesis, the latter set belongs to 
 $\mathscr P$), and consider the function $f_{B_F}:\omega\longrightarrow\omega$ as 
given by 
 Lemma~\ref{specialfunction}. Since there are 
 $\max\{|\alpha|,\omega\}<\mathfrak c$ 
 possible $F\subseteq\{X_\xi\big|\xi\lt\alpha\}$, by $\mathfrak d=\mathfrak c$ we can 
 choose a $g:\omega\longrightarrow\omega$ such that for each of the $F$, we have 
 $g\not\leq^*f_{B_F}$. We now let $X_\alpha=B(g)$. Since in particular 
 $g\not\leq^*f_{B}$ (since $B=B_\varnothing$), $X_\alpha$ will be a perfect 
 unbounded subset of $B$, and for every finite 
 $F\subseteq\{X_\xi\big|\xi\lt\alpha\}$, we notice that 
 \begin{equation*}
  X_\alpha\cap\left(\bigcap_{\xi\in F}X_\xi\right)\supseteq B(g)\cap\left(\bigcap_{\xi\in F}X_\xi(g)\right)=\left(B\cap\left(\bigcap_{\xi\in F}X_\xi\right)\right)(g)\supseteq B_F(g),
 \end{equation*}
 and the rightmost set above is crowded unbounded. Thus our $X_\alpha$ can be added 
 to the previously obtained $X_\xi$ while preserving the three required properties, 
 and so the recursive construction can continue. In the end we  let 
 $u$ be the filter generated by $\{X_\alpha\big|\alpha\lt\mathfrak c\}$, which 
 will clearly be a gruff ultrafilter.
 \end{proof}

\section{Gruff ultrafilters in the Random real model}

We will now proceed to prove that there are gruff ultrafilters in the Random real model. 
We will closely follow a strategy  used by Paul E. 
Cohen (not to 
be confused with the Paul J. Cohen 
that discovered forcing!) in his proof that there are P-points in the 
Random 
model~\cite{cohen}. For this purpose,
we introduce the following definition.

\begin{definition}\label{defpathway}
A continuous increasing 
$\omega_1$-sequence $\langle A_\alpha\big|\alpha\lt\omega_1\rangle$
of reals (elements of $\omega^\omega$) will be called a 
\textbf{strong pathway} if it satisfies the following three properties:
\begin{enumerate}
 \item $\bigcup\limits_{\alpha<\omega_1}A_\alpha=\omega^\omega$,
 \item for every $\alpha<\beta<\omega_1$ there exists a $g\in A_\beta$ which 
  is not dominated by any $f\in A_\alpha$.
 \item for every $\alpha\lt\omega_1$, if 
  $f_1,\ldots,f_n\in A_\alpha$ and $f\in\omega^\omega$ is definable 
  from $f_1,\ldots,f_n$, then $f\in A_\alpha$ (i.e. $A_\alpha$ is closed 
  under set-theoretic definability),
\end{enumerate}
\end{definition}

Cohen~\cite{cohen} defined a pathway to be a continuous, increasing 
$\omega_1$-sequence satisfying requirements (1) and (2) from 
Definition~\ref{defpathway}, along with the requirement (strictly weaker 
than (3) in the aforementioned Definition) that each $A_\alpha$ is 
closed under joins (the join of $f,g$ is the function mapping 
$2n$ to $f(n)$ and $2n+1$ to $g(n)$) and Turing-reducibility. In other 
words, in a pathway, we can only ensure that $f\in A_\alpha$ if 
it is explicitly (algorithmically) computable (in finitely many steps) from 
$f_1,\ldots,f_n\in A_\alpha$. Since this weak closure will not be enough 
for our purposes, we adapted Cohen's definition by demanding closure 
under any set-theoretic construction.

Now our proof will be done in two main steps. The first step is proving 
that there is a strong pathway in the Random real model, and the second step is 
showing that the existence of a strong pathway implies the existence of 
a gruff ultrafilter. Each of the following two theorems realizes each of 
these two steps.

\begin{theorem}\label{existspathway}
Take a ground model $\V$ that satisfies $\ch$, let $\lambda$ be 
a cardinal in $\V$, and let $\mathcal R_\lambda$ be the forcing 
that adds $\lambda$ many Random reals (this is, $\mathcal R_\lambda$ 
consists of all non-null sets in the measure algebra of $2^\lambda$, 
ordered by inclusion modulo null sets). If $r:\lambda\longrightarrow2$ 
is the sequence of Random reals added by $\mathcal R_\lambda$ (this is, 
if $r$ is the unique element that belongs to the 
intersection of an $(\mathcal R_\lambda,\V)$-generic filter), then 
in $\V[r]$ there exists a strong pathway.
\end{theorem}

\begin{proof}
We will start working in the ground model $\V$. Since $\V\vDash\ch$, we can 
fix an enumeration 
$\langle f_\alpha\big|\alpha<\omega_1\rangle$ of $\omega^\omega$. We now choose 
a large enough cardinal $\theta$, and construct a continuous increasing 
$\omega_1$-chain of countable elementary submodels 
$M_0\prec M_1\prec\cdots\prec M_\alpha\prec\cdots\prec H(\theta)$ satisfying, 
for every $\alpha\lt\omega_1$, that $f_\alpha\in M_{\alpha+1}$ and also that 
some $f\in M_{\alpha+1}$ dominates all elements of $\omega^\omega\cap M_\alpha$.
We note that, if $\pi:\omega\longrightarrow\lambda$ is an injective function 
from $\V$, then $r\circ\pi:\omega\longrightarrow 2$ is a Random real over 
$\V$ (since it will avoid any Borel null set whose code belongs to $\V$), and 
hence also over each of the $M_\alpha$ (since Borel codes that belong to $M_\alpha$ 
also belong to $\V$, and by elementarity). We therefore define 
the $A_\alpha$ to be given by
\begin{equation*}
A_\alpha=\bigcup_{\pi:\omega\longrightarrow\lambda \atop 
\pi\in\V\text{ is injective}}\omega^\omega\cap M_\alpha[r\circ\pi].
\end{equation*}
We now proceed to prove that the sequence 
$\langle A_\alpha\big|\alpha\lt\omega_1\rangle$ (which clearly is 
a continuous increasing sequence) satisfies all three requirements 
of the definition of a strong pathway.
\begin{enumerate}
 \item Given a real $f\in\V[r]$, the structure of $\mathcal R_\lambda$ 
 as a measure algebra, along with the fact that it is a c.c.c. forcing 
 notion, imply that $f$ can be computed from countably many bits from $r$, 
 along with countably many reals. In other words, there are 
 $f_0,f_1,\ldots,f_n,\ldots\in\V$ and there is an injection 
 $\pi:\omega\longrightarrow\lambda$ such that, if 
 $\alpha\lt\omega_1$ is large enough that all $f_i\in M_\alpha$, 
 then $f\in M_\alpha[r\circ\pi]$. This proves that 
 $\bigcup_{\alpha\lt\omega_1}A_\alpha=\omega^\omega$.
 \item Let $\alpha\lt\beta\lt\omega_1$ and $\pi:\omega\longrightarrow\lambda$ be an 
 injection. Since $\mathcal R_\lambda$ is an $\omega^\omega$-bounding forcing 
 notion, every element of $M_\alpha[r\circ\pi]$ (which is a Random extension 
 of $M_\alpha$) is dominated by some real in $M_\alpha$, and all reals in 
 $M_\alpha$ are in turn dominated by a single real $g\in M_\beta$. Hence 
 this $g$ dominates all elements of $A­_\alpha$ (in particular, $g$ is not 
 dominated by any $f\in A_\alpha$).
 \item Let $f$ be set-theoretically definable from $f_1,\ldots,f_n\in A_\alpha$. 
 There are injective functions $\pi_i:\omega\longrightarrow\lambda$, for 
 $1\leq i\leq n$, such that $f_i\in M_\alpha[r\circ\pi_i]$. For each nonempty 
 subset $a\subseteq n$, define
 \begin{equation*}
  \Lambda_a=\left(\bigcap_{i\in a}\ran(f_i)\right)\setminus\left(\bigcup_{j\notin a}\ran(f_j)\right)\subseteq\lambda.
 \end{equation*}
 Note that the $\Lambda_a$ are all countable, and pairwise disjoint. Now, 
 working in $\V$, pick an effective partition of $\omega$ into $2^n-1$ cells 
 $N_a$ (indexed by the nonempty subsets $a\subseteq n$), with each $N_a$ of cardinality 
 $|\Lambda_a|$, and let $\pi:\omega\longrightarrow\lambda$ be an injection 
 which bijectively maps each $N_a$ onto $\Lambda_a$. Since the $\Lambda_a$ are 
 definable, they inhabit $M_\alpha$, and so do the sets 
 $R_i=\bigcup_{i\in a}N_a$, which 
 allow us to reconstruct $r\circ\pi_i$ from $r\circ\pi$ 
 (note that $\pi[R_i]=\ran(f_i)$). Thus for each $i$ we have that 
 \begin{equation*}
  f_i\in M_\alpha[r\circ\pi_i]\subseteq M_\alpha[r\circ\pi],
 \end{equation*}
 and therefore it must be the case that $f\in M_\alpha[r\circ\pi]\subseteq A_\alpha$, 
 so $A­_\alpha$ has the required closure, and we are done.
\end{enumerate}

\end{proof}

In the proof of the following theorem, we shall be 
using the same effective enumeration $\{q_n\big|n<\omega\}$ of 
$\mathbb Q$ that we used in the previous section, along with 
the definitions of the $J_n^f$ and $X(f)$ for $n\lt\omega$, 
$f:\omega\longrightarrow\omega$, and $X\subseteq\mathbb Q$. 
Also, given a set $X\in\mathscr P$, we will occasionally
invoke the function $f_X$ whose 
existence is guaranteed by Lemma~\ref{specialfunction}.
Finally, we will use the fact that for every 
$X\subseteq\mathbb Q$, if $Y$ is the maximal crowded subset 
of $X$ ($Y$ can be defined to be the union of all crowded 
subsets of $X$, or alternatively, the end result of the 
transfinite Cantor-Bendixson decomposition of $X$), then 
the function $\chi_{\{n\lt\omega\big|q_n\in Y\}}$ is 
definable from $\chi_{\{n\lt\omega\big|q_n\in X\}}$.

\begin{theorem}
Let $\lambda$ be any infinite cardinal number in a ground model 
$\V$ that satisfies $\ch$, and let $\mathcal R_\lambda$ be the forcing 
notion that adds $\lambda$ many Random reals to $\V$. If 
$r:\lambda\longrightarrow 2$ is the generic function
added by $\mathcal R_\lambda$, 
then in $\mathbf{V}[r]$ there is a gruff ultrafilter.
\end{theorem}

\begin{proof}
By Theorem~\ref{existspathway}, we can let 
$\langle A_\alpha\big|\alpha\lt\omega_1\rangle$ be a strong pathway in 
$\V[r]$.
For each $\alpha$, choose a function $g_\alpha\in A_{\alpha+1}$ 
which is not 
dominated by any element of $A_\alpha$. We will recursively construct a 
continuous increasing sequence 
of filters 
$\mathcal F_\alpha\subseteq\mathscr P$, satisfying the following for every 
$\alpha\lt\omega_1$:
\begin{enumerate}
\item $\mathcal F_\alpha$ has a basis of perfect unbounded sets $P\subseteq\mathbb Q$ 
 such that 
\begin{equation*}
\chi_{\{n\lt\omega\big|q_n\in P\}}\in A_\alpha,
\end{equation*}

\item for every $X\subseteq\mathbb Q$ such that 
$\chi_{\{n\lt\omega\big|q_n\in X\}}\in A_\alpha$, either $X$ or 
 $\mathbb Q\setminus X$ belongs to $\mathcal F_{\alpha+1}$.
\end{enumerate}
If we succeed in performing such a construction, then clearly 
$u=\bigcup\limits_{\alpha\lt\omega_1}\mathcal F_\alpha$ will be an ultrafilter 
(by condition (2)), 
with a basis of perfect unbounded sets (by condition (1)), i.e. a gruff ultrafilter.

We start off with $\mathcal F_0=\{\mathbb Q\}$, and when $\alpha$ is a limit 
ordinal we  let $\mathcal F_\alpha=\bigcup_{\beta<\alpha}\mathcal F_\beta$. We only 
need to show that the $\mathcal F_\alpha$ thus defined satisfies condition (1) (note that condition 
(2) is vacuous in this case, since it only concerns successor cardinals). To this effect, let 
$X\in\mathcal F_\alpha$, which means that $X\in\mathcal F_\beta$ for some $\beta<\alpha$.  Now by 
our inductive hypothesis (1), there exists some $P\in\mathcal F_\beta\subseteq\mathcal F_\alpha$ 
such that $P\subseteq X$ and $\chi_{\{n<\omega\big|q_n\in P\}}\in A_\beta\subseteq A_\alpha$, which 
shows that condition (1) still holds of $\mathcal F_\alpha$.

Now assume that we have already constructed $\mathcal F_\alpha$ satisfying 
conditions (1) and (2). In order to construct $\mathcal F_{\alpha+1}$, we first 
extend $\mathcal F_\alpha$ to some ultrafilter $\mathcal G\subseteq\mathscr P$, 
and let $\mathcal F_{\alpha+1}$ be 
the upward closure of the family
\begin{equation*}
\{X(g_\alpha)\big|X\in\mathcal G\text{, }
\chi_{\{n\lt\omega\big|q_n\in X\}}\in A_{\alpha}\text{ and }X\text{ is crowded unbounded}\}.
\end{equation*}

First of all, notice that  each of the $X(g_\alpha)$ generating $\mathcal F_{\alpha+1}$ is a 
perfect unbounded set (hence $\mathcal F_{\alpha+1}\subseteq\mathscr P$), since 
$\chi_{\{n\lt\omega\big|q_n\in X\}}\in A_\alpha$ implies that the $f_X$ given 
by Lemma~\ref{specialfunction} is also an element of $A_\alpha$, hence 
$g_\alpha\not\leq^*f_X$ and so $X(g_\alpha)$ is crowded unbounded. This reasoning also helps us 
establish that $\mathcal F_{\alpha+1}$ is indeed a filter, since 
whenever $Y_1,\ldots,Y_n\in\mathcal F_{\alpha+1}$ it is because for some 
crowded unbounded $X_1,\ldots,X_n\in\mathcal G$ with 
$\chi_{\{n\lt\omega\big|q_n\in X_i\}}\in A_\alpha$, we have that 
$X_i(g_\alpha)\subseteq Y_i$. But if we let $X$ be the maximal crowded unbounded 
subset of $X_1\cap\cdots\cap X_n$ (which exists since this intersection belongs 
to $\mathcal G\subseteq\mathscr P$), then it will be the case that $X\in\mathcal G$, 
$\chi_{\{n\lt\omega\big|q_n\in X\}}\in A_\alpha$, 
and therefore 
\begin{equation*}
Y_1\cap\cdots\cap Y_n\supseteq X_1(g_\alpha)\cap\cdots\cap X_n(g_\alpha)=(X_1\cap\cdots\cap X_n)(g_\alpha) 
\supseteq X(g_\alpha)\in\mathscr F_{\alpha+1}.
\end{equation*}

 Now, in order to show that $\mathcal F_{\alpha+1}$ extends $\mathcal F_{\alpha}$, we let 
 $X\in\mathcal F_\alpha$. By condition (1) on our induction hypothesis, there exists a perfect set 
 $P\subseteq X$ such that $\chi_{\{n<\omega\big|q_n\in P\}}\in A_\alpha$ and $P\in\mathcal F_\alpha$. 
 Therefore we have that
$X\supseteq P\supseteq P(g_\alpha)\in\mathcal F_{\alpha+1}$, and so 
$\mathcal F_\alpha\subseteq\mathcal F_{\alpha+1}$.

We now notice also that all of the $X(g_\alpha)$ that generate the filter $\mathcal F_{\alpha+1}$ 
are such that $\chi_{\{n\lt\omega\big|q_n\in X(g_\alpha)\}}\in A_{\alpha+1}$, 
since each $X(g_\alpha)$ is definable from 
$\chi_{\{n\lt\omega\big|q_n\in X\}}\in A_\alpha\subseteq A_{\alpha+1}$ and 
$g_\alpha\in A_{\alpha+1}$ (and 
our fixed enumeration of $\mathbb Q$, which lies in $A_0\subseteq A_{\alpha+1}$). Hence our 
$\mathcal F_{\alpha+1}$ satisfies condition (1) of the construction.

To show that condition (2) is also satisfied, let $X\subseteq\mathbb Q$ be such that 
$\chi_{\{n\lt\omega\big|q_n\in X\}}\in A_\alpha$, and let us show that 
$\mathcal F_{\alpha+1}$ contains one of $X,\mathbb Q\setminus X$. 
Since $\mathcal G$ is an ultrafilter, there is $Y\in\{X,\mathbb Q\setminus X\}$ such 
that $Y\in\mathcal G$. Since $\mathcal G\subseteq\mathscr P$, if we let $Z$ be the 
maximal crowded unbounded subset of $Y$, then $Z\in\mathcal G$. But notice that 
$\chi_{\{n\lt\omega\big|q_n\in Z\}}\in A_\alpha$, and this implies that 
$Y\supseteq Z\supseteq Z(g_\alpha)\in\mathcal F_{\alpha+1}$.

This finishes the recursive construction, and we are done.
\end{proof}

The authors would like to thank Andreas Blass for 
some very fruitful discussions, particularly concerning Cohen's 
paper~\cite{cohen}, as well as for greatly simplifying the proof of 
Lemma~\ref{boundedvsunbounded}.

\end{document}


{\bf Corrigendum}

\title[Corrigendum to ``Gruff ultrafilters'']{Corrigendum to ``Gruff ultrafilters'' [Topol. Appl. 210 (2016) 355--365]}

\begin{nouppercase}
\maketitle
\end{nouppercase}

The authors regret that in Section~4 of the paper~\cite{gruff-article}, the proof of Theorem~4.2 contains a mistake, and the authors do not know whether it can be fixed. The problem is that the terms $A_\alpha$ of the strong pathway $\langle A_\alpha\big|\alpha<\omega_1\rangle$, as defined in the paper, will not be closed under binary set-theoretically definable operations; in fact, they will not even be closed under the join operation (the authors are very grateful to Osvaldo Guzm\'an for pointing this out, as well as for providing a concrete example of the failure of this closure property). In fact, the same problem arises in an analogous argument from the paper~\cite{cohen}, where a similar technique was first implemented. Consequently, the question of whether there are P-points in every Random extension of a model of $\ch$ (an affirmative answer to which~\cite{cohen} had been considered an established result for quite some time) appears to be open at the moment.

Since the proof of Theorem~4.3 in~\cite{gruff-article} depends on Theorem~4.2, the authors regrettably withdraw the claim made in the statement of this Theorem. We do not know if there exist gruff ultrafilters in every model obtained by adding Random reals to a model of $\ch$. We do, however, know that there is a model $M$ of $\ch$ such that there are gruff ultrafilters in any model obtained by adding any number of random reals to $M$. To see this it suffices to consider $M=\mathbf{V}^{\mathcal C}$, where $\mathbf{V}\vDash\ch$ and $\mathcal C$ denotes the forcing notion for adding $\omega_1$-many Cohen reals to $\mathbf{V}$. Let $\mathring{\mathcal R}$ be the ($\mathbf{V}^\mathcal C$-name for the) forcing notion that adds $\lambda$-many Random reals to $\mathbf{V}^{\mathcal C}$, where $\lambda$ is some cardinal of $\mathbf{V}^{\mathcal C}$ (equivalently, of $\mathbf{V}$, since $\mathcal C$ is a c.c.c. forcing notion). Then we claim that in the generic extension $\mathbf{V}^{\mathcal C\star\mathring{\mathcal R}}$ there is a strong pathway. 

To see this, let $\langle c_\alpha\big|\alpha<\omega_1\rangle$ be the $\omega_1$-sequence of Cohen reals added by $\mathcal C$. For each $\alpha\leq\omega_1$, consider the model of Set Theory $\mathbf{V}_\alpha=\mathbf{V}[\langle c_\xi\big|\xi<\alpha\rangle]$.
If we let $b^W$ denote the interpretation of the Borel code $b$ in the model of Set Theory $W$, and we have two such models $W\subseteq V$,  recall that a function $f\in 2^{\lambda}\cap V$  is \emph{random} over $W$ if $f\notin b^V$ for any null $G_\delta$-set coded in $W$ (i.e. such that $b\in W$). In particular, if $G$ is an $\mathcal R$-generic filter over $\mathbf{V}_{\omega_1}$, and $r_G(\alpha)=i$ if and only if $[\{f\in 2^{\lambda}: f(\alpha)=i\}]\in G$, then $r_G\in\mathbf{V}_{\omega_1}[G]$ is random over $\mathbf{V}_{\omega_1}$, and consequently it is also random over $\mathbf{V}_{\alpha}$ for every $\alpha<\omega_1$. It is then easy to verify that the sequence $\langle A_\alpha\big|\alpha<\omega_1\rangle$ given by
\begin{equation*}
A_{\alpha+1}=\omega^\omega\cap\mathbf{V}_\alpha[r_G]
\end{equation*}
(for limit $\alpha$, we let $A_\alpha=\bigcup_{\xi<\alpha}A_\xi$ in order to make the sequence continuous) will be a strong pathway:
First, as $\mathbf{V}_\alpha[r_G]$ is a generic extension of $\mathbf{V}$, $A_{\alpha+1}$ has all the closure properties required; moreover, $c_\alpha\in\mathbf{V}_{\alpha+1}\setminus\mathbf{V}_\alpha$ is not dominated by any real from $\mathbf{V}_\alpha$ and hence it is not dominated by any element of $A_{\alpha+1}$ either.
On the other hand, every real in $\mathbf{V}^{\mathcal C\star\mathring{\mathcal R}}$ only depends on $r_G$ and countably many of the 
$\langle c_\alpha\big|\alpha<\omega_1\rangle$, so 
$$\omega^\omega=\bigcup_{\alpha<\omega_1}A_\alpha.$$
The proof of Theorem 4.3 goes through verbatim upon replacing the Random model with the model $\mathbf{V}_{\omega_1}[G]$, by using the strong pathway defined above. Hence we can conclude the existence of gruff ultrafilters in any model that arises from having $\ch$, adding $\omega_1$ Cohen reals, and subsequently adding some uncountable number of Random reals. (Incidentally, with the same reasoning we can conclude that there are P-points in any such model, which is an unpublished old result of Kunen).

The authors would like to apologise for any inconvenience caused.